\documentclass[11pt,a4paper]{amsart}
\usepackage{amsthm, amssymb, amsmath, a4wide}
\usepackage[all]{xy} 
\usepackage{paralist} 
\usepackage{graphicx} 
\usepackage[sans]{dsfont} 
\usepackage{varioref} 
\usepackage{subfigure}
\usepackage{tikz}
\usepackage{pgf}
\usetikzlibrary{decorations.pathreplacing,calc}
\usepackage{xifthen}
\usepackage[pagebackref,colorlinks,linkcolor=red,citecolor=blue,urlcolor=blue,hypertexnames=true]{hyperref}
\usepackage{array}

\newcommand{\xomegatheta}{\bar X_{\vec \omega,\vec \theta}}
\newcommand{\scrxomega}{\scrX_{\vec \omega}}

\DeclareMathOperator{\ind}{ind}

\DeclareMathOperator\ord{ord}

\DeclareMathOperator\proj{Proj}

\DeclareMathOperator\sing{Sing} 
\DeclareMathOperator\spec{Spec}

\newcommand{\scrV}{\ensuremath{\mathcal{V}}}
\newcommand{\scrX}{\ensuremath{\mathcal{X}}}

\newcommand{\cc}{\ensuremath{\mathbb{C}}}

\newcommand{\pp}{\ensuremath{\mathbb{P}}}

\newcommand{\zz}{\ensuremath{\mathbb{Z}}}

\newcommand{\WP}{\ensuremath{{\bf W}\pp}}

\newtheorem{thm}{Theorem}[section]
\newtheorem*{thm*}{Theorem}

\newtheorem*{lemma*}{Lemma}
\newtheorem{lemma-in-thm}{Lemma}[thm]

\newtheorem{prop}[thm]{Proposition}
\newtheorem*{prop*}{Proposition}
\newtheorem{cor}[thm]{Corollary}

\newtheorem*{claim*}{Claim}

\newtheorem*{conjecture*}{Conjecture}

\theoremstyle{definition} 
\newtheorem{conjecture}[thm]{Conjecture}
\newtheorem{defn}[thm]{Definition}
\newtheorem*{defn*}{Definition}

\newtheorem*{definotation*}{Definition-Notation}
\newtheorem{example}[thm]{Example}
\newtheorem*{example*}{Example}

\newtheorem*{fact*}{Fact}

\newtheorem*{facts*}{Facts}

\newtheorem{notation}[thm]{Notation}

\newtheorem*{bold-note*}{Note}
\newtheorem{bold-question}[thm]{Question}
\newtheorem{rem}[thm]{Remark}

\newtheorem*{reminition*}{Remark-Definition}

\newtheorem*{remexample*}{Remark-Example}

\newtheorem*{remtation*}{Remark-Notation}

\newtheorem*{remuestion*}{Remark-Question}

\newtheorem*{constrinition*}{Construction-Definition}

\theoremstyle{remark}
\newtheorem*{rem*}{Remark}
\newtheorem*{note*}{Note}
\newtheorem*{notation*}{Notation}
\newtheorem*{question*}{Question}

\newtheorem*{questions*}{Questions}

\author{Pinaki Mondal}
\address{Weizmann Institute of Sciences, Israel}
\title{Normal analytic compactifications of $\cc^2$}

\begin{document}

\begin{abstract}
This is a survey of some results on the structure and classification of normal analytic compactifications of $\cc^2$. Mirroring the existing literature, we especially emphasize the compactifications for which the curve at infinity is irreducible. 
\end{abstract}

\maketitle

\section{Introduction}
A compact normal analytic surface $\bar X$ is called a compactification of $\cc^2$ if there is a subvariety $C$ (the {\em curve at infinity}) such that $\bar X \setminus C$ is isomorphic to $\cc^2$. {\em Non-singular} compactifications of $\cc^2$ have been studied at least since 1954 when Hirzebruch included the problem of finding all such compactifications in his list of problems on differentiable and complex manifolds \cite{hirzebruch-problems}. Remmert and Van de Ven \cite{remmert-van-de-ven} proved that $\pp^2$ is the only non-singular analytic compactification of $\cc^2$ for which the curve at infinity is irreducible. Kodaira as part of his classification of surfaces, and independently Morrow \cite{morrow} showed that every non-singular compactification of $\cc^2$ is {\em rational} (i.e.\ bimeromorphic to $\pp^2$) and can be obtained from $\pp^2$ or some Hirzebruch surface via a sequence of blow-ups. Moreover, Morrow \cite{morrow} gave the complete classification (modulo extraneous blow-ups of points at infinity) of non-singular compactifications of $\cc^2$ for which the curve at infinity has normal crossing singularities. \\

The main topic of this article is therefore {\em singular} normal analytic compactifications of $\cc^2$. The studies on singular normal analytic compactifications so far have concentrated mostly on the (simplest possible) case of compactifications for which the curve at infinity is irreducible; following \cite{ohta}, we call these {\em primitive} compactifications (of $\cc^2$). These were studied from different perspectives in \cite{brentonfication}, \cite{brenton-singular}, \cite{brenton-graph-1}, \cite{miyanishi-zhang}, \cite{furushima}, \cite{ohta}, \cite{kojima}, \cite{koji-hashi}, and more recently in \cite{sub2-1}, \cite{contractibility} and \cite{sub2-2}. The primary motive of this article is to describe these results. For relatively more technical of the results, however, we  omit the precise statements and prefer to give only a `flavour'. The only new results of this article are Proposition \ref{curve-at-infinity-1} and parts of Proposition \ref{singular-proposition}.

\begin{notation}
Unless otherwise stated, by a `compactification' we mean throughout a normal analytic compactification of $\cc^2$.
\end{notation}

\section{Analytic vs.\ algebraic compactifications}
As mentioned in the introduction, non-singular compactifications of $\cc^2$ are projective, and therefore, {\em algebraic} (i.e.\ analytifications of proper schemes). In particular this implies that every compactification $\bar X$ of $\cc^2$ is necessarily {\em Moishezon}, or equivalently, analytification of a proper {\em algebraic space}. Moreover, if $\pi: \bar X' \to X$ is a resolution of singularities of $\bar X$, then the intersection matrix of the curves contracted by $\pi$ is negative definite. On the other hand, by the contractibility criterion of Grauert \cite{grauert}, for every non-singular compactification $\bar X'$ of $\cc^2$ and a (possibly reducible) curve $C \subseteq \bar X' \setminus \cc^2$ with negative definite intersection matrix, there is a compactification $\bar X$ of $\cc^2$ and a birational holomorphic map $\pi: \bar X' \to \bar X$ such that $\pi$ contracts only $C$ (and no other curve). The preceding observation, combined with the classification of non-singular compactifications of $\cc^2$ due to Kodaira and Morrow, forms the basis of our understanding of (normal) compactifications of $\cc^2$. However, it is an open question how to determine if a (singular) compactification of $\cc^2$ constructed via contraction of a given (possibly reducible) negative definite curve (from a non-singular compactification) is algebraic. \cite{contractibility} solves this question in the special case of {\em primitive} compactifications of $\cc^2$ (for which, in particular, algebraicity is equivalent to projectivity - see Theorem \ref{projective-embedding}).\\

More precisely, let $X := \cc^2$ and $\bar X^0 := \pp^2 \supseteq X$. Let $\bar X$ be a primitive compactification of $X$ which is not isomorphic to $\pp^2$ and $\sigma: \bar X^0 \dashrightarrow \bar X$ be the bimeromorphic map induced by identification of $X$. Then $\sigma$ maps the {\em line at infinity} $L_\infty := \pp^2 \setminus X$ (minus the points of indeterminacy) to a point $P_\infty \in C_\infty := \bar X \setminus X$.

\begin{thm}[{\cite[Corollary 1.6]{contractibility}}] \label{algebraic-place}
$\bar X$ is algebraic iff there is an algebraic curve $C \subseteq X$ with one place at infinity\footnote{Recall that $C$ has {\em one place at infinity} iff $C$ meets the line at infinity at only one point $Q$ and $C$ is unibranch at $Q$.} such that $P_\infty$ does {\em not} belong to the closure of $C$ in $\bar X$. 
\end{thm}

\begin{rem}\label{effective-remark}
Theorem \ref{algebraic-place} can be viewed as the {\em effective} version of (a special case of) some other algebraicity criteria (e.g.\ those of \cite{schroe-traction}, \cite{palka-Q1}). More precisely, in the situation of Theorem \ref{algebraic-place}, both \cite[Theorem 3.3]{schroe-traction} and \cite[Lemma 2.4]{palka-Q1} imply that $\bar X$ is algebraic iff there is an algebraic curve $C \subseteq X$ which satisfies the following (weaker) condition:
\begin{align*}
\text{$P_\infty$ does not belong to the closure of $C$ in $\bar X$.} \tag{$*$} \label{non-intersecting-property} 
\end{align*}
Theorem \ref{algebraic-place} implies that in the algebraic case it is possible to choose $C$ with an additional property, namely that it has one place at infinity. A possible way to construct such curves is via the {\em key forms} of the {\em divisorial valuation} on $\cc(X)$ associated with $C_\infty$ (see Remarks \ref{approximate-remark} and \ref{general-key-remark}). The key forms are in general not polynomials, but if they are indeed polynomials, then the {\em last} key form defines a curve $C$ with one place at infinity which satisfies \eqref{non-intersecting-property}. On the other hand, if the last key form is not a polynomial, then it turns out that there are no curve $C \subseteq X$ which satisfies \eqref{non-intersecting-property} \cite[Proposition 4.2]{contractibility}, so that $\bar X$ is not algebraic.
\end{rem}

\begin{example}[{\cite[Examples 1.3 and 2.5]{contractibility}}]\label{non-example}
Let $(u,v)$ be a system of `affine' coordinates near a point $O \in \pp^2$ (`affine' means that both $u=0$ and $v=0$ are lines on $\pp^2$) and $L$ be the line $\{u=0\}$. Let $C_1$ and $C_2$ be curve-germs at $O$ defined respectively by $f_1 := v^5 - u^3$ and $f_2 := (v-u^2)^5 - u^3$. For each $i,r$, $1 \leq i \leq 2$ and $r \geq 0$, let $\tilde X_{i,r}$ be the surface constructed by resolving the singularity of $C_i$ at $O$ and then blowing up $r$ more times the point of intersection of the (successive) strict transform of $C_i$ with the exceptional divisor. Let $\tilde E^{(i,r)}$ be the union of the strict transform $\tilde L_{i,r}$ (on $\tilde X_{i,r}$) of $L$ and (the strict transforms of) all exceptional curves except the exceptional curve $E^*_{i,r}$ for the {\em last} blow up. It is straightforward to compute that for $r \leq 9$ the intersection matrix of $\tilde E^{(i,r)}$ is negative definite, so that $\tilde E^{(i,r)}$ can be analytically contracted to the unique singular point $P_{i,r}$ on a normal surface $\bar X_{i,r}$ which is a primitive compactification of $\cc^2$. Note that the weighted dual graphs of $\tilde E^{(i,r)} \cup E^*_{i,r}$ are {\em identical} (see Figure \ref{fig:non-example}). \\

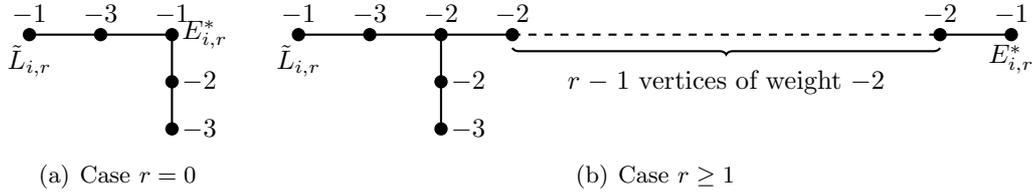
\begin{figure}[htp]
\begin{center}

\subfigure[Case $r = 0$]{
\begin{tikzpicture}[scale=1.25, font = \small] 	
 	\pgfmathsetmacro\edge{.75}
 	\pgfmathsetmacro\vedge{.5}
	
 	\draw[thick] (-2*\edge,0) -- (0,0);
 	\draw[thick] (0,0) -- (0,-2*\vedge);

 	\fill[black] ( - 2*\edge, 0) circle (2pt);
 	\fill[black] (-\edge, 0) circle (2pt);
 	\fill[black] (0, 0) circle (2pt);
 	\fill[black] (0, -\vedge) circle (2pt);
 	\fill[black] (0, - 2*\vedge) circle (2pt);
 	
 	\draw (-2*\edge,0)  node (e0up) [above] {$-1$};
 	\draw (-2*\edge,0 )  node (e0down) [below] {$\tilde L_{i,r}$};
 	\draw (-\edge,0 )  node (e1up) [above] {$-3$};	
 	\draw (0,0 )  node [above] {$-1$};	
 	\draw (0,0 )  node [right] {$E^*_{i,r}$};
 	\draw (0,-\vedge )  node (down1) [right] {$-2$};
 	\draw (0, -2*\vedge)  node (down2) [right] {$-3$};
 			 	
\end{tikzpicture}
}
\subfigure[Case $r \geq 1$]{
\begin{tikzpicture}[scale=1.25, font = \small] 	
 	\pgfmathsetmacro\dashedge{4.5}	
 	\pgfmathsetmacro\edge{.75}
 	\pgfmathsetmacro\vedge{.5}

 	\draw[thick] (-2*\edge,0) -- (\edge,0);
 	\draw[thick] (0,0) -- (0,-2*\vedge);
 	\draw[thick, dashed] (\edge,0) -- (\edge + \dashedge,0);
 	\draw[thick,] (\edge + \dashedge,0) -- (2*\edge + \dashedge,0);
 	
 	\fill[black] ( - 2*\edge, 0) circle (2pt);
 	\fill[black] (-\edge, 0) circle (2pt);
 	\fill[black] (0, 0) circle (2pt);
 	\fill[black] (0, -\vedge) circle (2pt);
 	\fill[black] (0, - 2*\vedge) circle (2pt);
 	\fill[black] (\edge, 0) circle (2pt);
 	\fill[black] (\edge + \dashedge, 0) circle (2pt);
 	\fill[black] (2*\edge + \dashedge, 0) circle (2pt);
 	
 	\draw (-2*\edge,0)  node (e0up) [above] {$-1$};
 	\draw (-2*\edge,0 )  node (e0down) [below] {$\tilde L_{i,r}$};
 	\draw (-\edge,0 )  node (e1up) [above] {$-3$};
 	\draw (0,0 )  node (middleup) [above] {$-2$};	
 	\draw (0,-\vedge )  node (down1) [right] {$-2$};
 	\draw (0, -2*\vedge)  node (down2) [right] {$-3$};
 	\draw (\edge,0)  node (e+1-up) [above] {$-2$};
 	\draw (\edge + \dashedge,0)  node (e-last-1-up) [above] {$-2$};
 	\draw (2*\edge + \dashedge,0)  node [above] {$-1$};
 	\draw (2*\edge + \dashedge,0)  node [below] {$E^*_{i,r}$};
 	
 	\draw [thick, decoration={brace, mirror, raise=5pt},decorate] (\edge,0) -- (\edge + \dashedge,0);
 	\draw (\edge + 0.5*\dashedge,-0.5) node [text width= 5cm, align = center] (extranodes) {$r-1$ vertices of weight $-2$};
 			 	
\end{tikzpicture}
}
\caption{Dual graph of $\tilde E^{(i,r)} \cup E^*_{i,r}$}\label{fig:non-example}
\end{center}
\end{figure}

Choose coordinates $(x,y) := (1/u,v/u)$ on $X := \pp^2\setminus L$. Then $C_1 \cap X = V(y^5 - x^2)$ and $C_2 \cap X = V((xy-1)^5 - x^7)$. Let $\tilde C_{i,r}$ (resp.\ $C_{i,r}$) be the strict transform of $C_i$ on $\tilde X_{i,r}$ (resp.\ $\bar X_{i,r}$). Note that each $C_{1,r}$ satisfies \eqref{non-intersecting-property}, so that all $\bar X_{1,r}$ are algebraic by the criteria of Schr\"oer and Palka. On the other hand, if $L'_{2,r}$ is the pullback on $\tilde X_{2,r}$ of a general line in $\pp^2$, then $\tilde C_{2,r}- 5L'_{2,r}$ intersects components of $\tilde E^{(2,r)}$ trivially and $E^*_{2,r}$ positively, so its positive multiples are the only candidates for total transforms of curves on $\bar X_{2,r}$ satisfying \eqref{non-intersecting-property}, provided the latter surface is algebraic. In other words, Schr\"oer and Palka's criteria imply that $\bar X_{2,r}$ is algebraic if and only if some positive multiple of $\tilde C_{2,r}- 5L'_{2,r}$ is numerically equivalent to an effective divisor. Theorem \ref{algebraic-place} implies that such a divisor does not exist for $r = 8,9$. Indeed, the sequence of key forms associated to (the divisorial valuation on $\cc(x,y)$ corresponding to) $E^*_{i,r}$ for $0 \leq r \leq 9$ are as follows \cite[Example 3.22]{contractibility}:
\begin{align*}
\parbox{1.65cm}{key forms for $E^*_{1,r}$} 
		 &= \begin{cases}
			x,y & \text{if}\ r = 0,\\			
			x,y, y^5 - x^2 & \text{if}\ r\geq 1.
			\end{cases}
		&& &
\parbox{1.65cm}{key forms for $E^*_{2,r}$}
		 &= \begin{cases}
			x,y & \text{if}\ r = 0,\\			
			x,y, y^5 - x^2 & \text{if}\ 1 \leq r\leq 7, \\
			x,y, y^5 - x^2, y^5 - x^2 - 5y^4x^{-1} & \text{if}\ 8 \leq r\leq 9.
			\end{cases}
\end{align*}			
In particular, for $8 \leq r \leq 9$, the last key form for $E^*_{2,r}$ is {\em not} a polynomial. It follows (from Remark \ref{effective-remark} and Theorem \ref{algebraic-place}) that $\bar X_{2,r}$ are algebraic for $r \leq 7$, but $\bar X_{2,8}$ and $\bar X_{2,9}$ are {\em not} algebraic. On the other hand, the key forms for $E^*_{1,r}$ are polynomials for each $r$, $0 \leq r \leq 9$, which implies via the same arguments that $\bar X_{1,r}$ are algebraic, as we have already seen via Schr\"oer and Palka's criteria.
\end{example}

\begin{rem} \label{simplest-sing}
It can be shown (by explicitly computing the geometric genus and multiplicity) that the singularities at $P_{i,8}$ (of Example \ref{non-example}) are in fact {\em hypersurface singularities} which are {\em Gorenstein} and {\em minimally elliptic} (in the sense of \cite{laufer-elliptic}). Minimally elliptic Gorenstein singularities have been extensively studied, and in a sense they form the simplest class of non-rational singularities. Since having only rational singularities implies algebraicity of the surface (via a result of Artin), it follows that the non-algebraic surface $\bar X_{2,8}$ of Example \ref{non-example} is a normal non-algebraic Moishezon surface with the `simplest possible' singularity. 
\end{rem} 

We do not know to what extent the properties of $C$ and $\bar X$ of Theorem \ref{algebraic-place} influence one another (in the case that $\bar X$ is algebraic). It is not hard to see that $\bar X \setminus \{P_\infty\}$ has at most one singular point, and the singularity, if exists, is a cyclic quotient singularity (Proposition \ref{singular-proposition}). The following question was suggested by Tommaso de Fernex.

\begin{bold-question} \label{related-question}
Let $\bar X$ be a primitive algebraic compactification of $\cc^2$ formed by (minimally) resolving the singularities of a curve-germ at a point on the line $L_\infty$ at infinity on $\pp^2$ and then contracting the strict transform of $L_\infty$ and all exceptional curves other than the last one. Let $P_\infty$ be as in Theorem \ref{algebraic-place} and $g$ be the smallest integer such that there exists a curve $C$ on $\bar X$ with geometric genus $g$ which does not pass through $P_\infty$. What is the relation between $g$ and the singularity of $\bar X$ at $P_\infty$?  
\end{bold-question}

Some computed examples suggest the following conjectural answer to the first case of Question \ref{related-question}: 
\begin{conjecture} \label{conjecture}
In the situation of Question \ref{related-question}, $g = 0$ iff the singularity at $P_\infty$ is rational. 
\end{conjecture} 

A motivation behind Conjecture \ref{conjecture} is to understand the relation between rational singularity at a point and existence of rational curves that do not pass through the singularity, as discovered e.g.\ in \cite[Theorem 0.3]{flenndenberg-rational}. Another motivation is Abhyankar's question about the relation between the genus and {\em semigroup of poles} of plane curves with one place at infinity \cite[Question 3]{sathaye}. More precisely, if $\bar X$ and $C$ are as in Question \ref{related-question}, then the condition that $\bar X$ has a rational singularity induces (via assertion \eqref{rational-assertion} of Corollary \ref{singularity-cor}) a restriction on the semigroup of poles of $C$ (cf.\ Remark \ref{approximate-remark}). In particular, if Conjecture \ref{conjecture} is true, then it (together with Corollary \ref{singularity-cor}) will answer the genus zero case of Abhyankar's question.

\section{Curve at infinity}
Let $\bar X$ be a normal compactification of $X := \cc^2$ and $C_\infty := \bar X \setminus X$ be the curve at infinity. An application of the classification results of non-singular compactifications of $\cc^2$ to the desingularization of $\bar X$ immediately yields that $C_\infty$ is a connected {\em tree} of (possibly singular) rational curves. In this section we take a deeper look at the structure of $C_\infty$ and describe a somewhat stronger version of a result of Brenton \cite{brentonfication}.\\

Let $\Gamma_1, \ldots, \Gamma_k$ be the irreducible components of $C_\infty$. Choose a copy $\bar X^0$ of $\pp^2$ such that the center (i.e.\ image under the natural bimeromorphic map $\bar X \dashrightarrow \bar X^0$ induced by identification of $X$) of each $\Gamma_j$ on $\bar X^0$ is a point $O_j \in L_\infty$, where $L_\infty := \bar X^0 \setminus X$ is the `line at infinity' on $\bar X^0$. Fix a $\Gamma_j$, $1 \leq j \leq k$. For each pair of (distinct) points $P_1, P_2$ on $\Gamma_j$, define a positive integer $m_j(P_1, P_2)$ as follows:
\begin{align*}
m_j(P_1,P_2)
        &:= \min\{i_{O_j}(C_1,C_2):\ \text{for each $i$, $1 \leq i \leq 2$, $C_i$ is an analytic curve germ} \\
        &\mbox{\phantom{$:=\min\{$}} ~  \text{at $O_j$ distinct from (the germ of) $L_\infty$ and the closure of the strict } \\
        &\mbox{\phantom{$:=\min\{$}} ~  \text{transform of $C_i$ on $\bar X$ passes through $P_i$}\},\ \text{where}  \\
i_{O_j}(C_1,C_2) &:= \text{intersection multiplicity of $C_1$ and $C_2$ at $O_j$}.
\end{align*}
It is not hard to see (e.g.\ using \cite[Proposition 4.2]{sub2-1}) that there exists an integer $\tilde m_j$ and a unique point $\tilde P_j \in \Gamma_j$ such that 
\begin{enumerate}
\item $m_j(P_1,P_2) = \tilde m_j$ for all $P_1, P_2 \in \Gamma_j\setminus\{\tilde P_j\}$, and
\item $m_j(\tilde P_j, P') < \tilde m_j$ for all $P' \in \Gamma_j\setminus\{\tilde P_j\}$.
\end{enumerate}

\begin{rem}
$\tilde P_j$ has the following interpretation in the language of the {\em valuative tree} \cite{favsson-tree}: the valuative tree $\scrV_j$ at $O_j$ is the space of all valuations centered at $O_j$ (which has a natural tree-structure rooted at $\ord_{O_j}$). The order of vanishing $\ord_{\Gamma_j}$ along $\Gamma_j$ is an element of $\scrV_j$ and the points on $\Gamma_j$ are in a one-to-one correspondence with the {\em tangent vectors} at $\ord_{\Gamma_j}$ \cite[Theorem B.1]{favsson-tree}. Then $\tilde P_j$ is the point on $\Gamma_j$ which corresponds to the (unique) tangent vector at $\ord_j$ which is represented by $\ord_{O_j}$.
\end{rem}
The result below follows from a combination of \cite[Proposition 4.2]{sub2-1} and \cite[Proposition 3.1]{sub2-2}. 
\begin{prop}[cf.\ {the Proposition in \cite{brentonfication}}]\label{curve-at-infinity-1}
\mbox{}
\begin{enumerate}
\item \label{containsC} $\Gamma_j \setminus \tilde P_j \cong \cc$.
\item \label{singular-position} Either $\tilde P_j$ is a singular point of $\bar X$ or $\tilde P_j \in \Gamma_i$ for some $i \neq j$.
\end{enumerate}
\end{prop}

\begin{rem}
Assertion \ref{containsC} of Proposition \ref{curve-at-infinity-1} implies that for every proper birational map $\tilde \Gamma_i \to \Gamma_i$, the pre-image of $\tilde P_i$ consists of only one point and $\tilde \Gamma_i$ is uni-branched at that point. In particular, in the language of \cite{brentonfication}, $\tilde \Gamma_i$ has a {\em totally extraordinary singularity} at $\tilde P_i$. Consequently, Proposition \ref{curve-at-infinity-1} strengthens the main result of \cite{brentonfication}.
\end{rem}

\begin{rem}
Assertion \ref{singular-position} implies in particular that if $\bar X$ is non-singular and $C_\infty$ is irreducible, then $C_\infty$ is non-singular as well. More precisely, a theorem of Remmert and Van de Ven in \cite{remmert-van-de-ven} states that in this scenario $\bar X$ is isomorphic to $\pp^2$. On the other hand, it was shown in \cite{brentonfication} that Proposition \ref{curve-at-infinity-1} together with Morrow's classification \cite{morrow} of `minimal normal compactifications'\footnote{$\bar X$ is a `minimal normal compactification' (in the sense of Morrow), or in modern terminology, a minimal SNC-compactification of $X := \cc^2$ iff 
\begin{inparaenum}[(i)]
\item $\bar X$ is non-singular, 
\item each $\Gamma_i$ is non-singular, 
\item $C_\infty$ has at most normal-crossing singularities, and
\item for all $\Gamma_i$ with self-intersection $-1$, contracting $\Gamma_i$ destroys some of the preceding properties.
\end{inparaenum} 
} of $\cc^2$ implies the theorem of Remmert and Van de Ven.
\end{rem}

\begin{rem}
If $C_\infty$ is not irreducible, then it is possible that some $\Gamma_i$ is singular, even if $\bar X$ is non-singular. One such example was constructed in \cite{brentonfication} for which $C_\infty$ has two irreducible components.
\end{rem}

For special types of compactifications one can say more about the curve at infinity. We say that a compactification $\bar X$ of $\cc^2$ is {\em minimal} if $\bar X$ does not dominate any other (normal analytic) compactification of $\cc^2$, or equivalently (by Grauert's theorem), if the self-intersection number of every irreducible component of $C_\infty$ is non-negative.

\begin{prop}[{\cite[Proposition 3.7]{sub2-1}, \cite[Corollary 3.6]{sub2-2}}] \label{curve-at-infinity-special}
\mbox{}
\begin{enumerate}
\item \label{minimal-at-infinity} If $\bar X$ is minimal, then there is a unique point $P_\infty \in C_\infty$ such that $\Gamma_i \cap \Gamma_j = \{P_\infty\}$ for all $i \neq j$. In particular, $\tilde P_i = P_\infty$ for all $i$. 
\item \label{primitive-at-infinity} If $\bar X$ is primitive algebraic, then $\Gamma_1 = C_\infty$ is non-singular off $\tilde P_1$, and it has at worst a (non-normal) {\em toric singularity} at $\tilde P_1$.
\end{enumerate}
\end{prop}

\section{Singular points}
As in the preceding section, let $\bar X$ be a normal compactification of $X := \cc^2$ and $C_\infty$ be the curve at infinity. In Proposition \ref{singular-proposition} below we give upper bounds for $|\sing(\bar X)|$ in the general case and in the case that $\bar X$ is a minimal compactification. Note that both of these upper bounds are sharp \cite[Examples 3.9 and 4.8]{sub2-1}. Moreover, it is not hard to see that the lower bound for $|\sing(\bar X)|$ in both cases is zero, i.e.\ for each $k \geq 1$, there are non-singular minimal compactifications of $\cc^2$ with $k$ irreducible curves at infinity. 

\begin{prop} \label{singular-proposition}
Assume that $C_\infty$ has $k$ irreducible components. Let $\sing(\bar X)$ be the set of singular points of $\bar X$.
\begin{enumerate}[(1)]
\item \label{non-minimal-bound}
\begin{enumerate}
\item \label{non-minimal-upper-bound} $|\sing(\bar X)| \leq 2k$.
\item \label{non-minimal-types} $\bar X$ has at most one singular point which is not {\em sandwiched}\footnote{Recall that an isolated singular point $P$ on a surface $Y$ is {\em sandwiched} if there exists a birational map $Y \to Y'$ such that the image of $P$ is non-singular. Sandwiched singularities are {\em rational} \cite[Proposition 1.2]{lipman}}.
\end{enumerate} 
\item Assume $\bar X$ is a minimal compactification. Then 
\begin{enumerate}
\item \label{minimal-upper-bound} $|\sing(\bar X)| \leq k +1$.
\item \label{minimal-configuration} Let $P_\infty$ be as in assertion \ref{minimal-at-infinity} of Proposition \ref{curve-at-infinity-special}. Then $\left|\sing(\bar X) \setminus \{P_\infty\}\right| \leq k$. Moreover, every point in $\sing(\bar X) \setminus \{P_\infty\}$ is a cyclic quotient singularity.
\end{enumerate}
\end{enumerate} 
\end{prop}

\begin{proof}
Assertions \ref{non-minimal-upper-bound} and \ref{minimal-upper-bound} and the first statement of assertion \ref{minimal-configuration} follows from \cite[Proposition 3.7]{sub2-1}. We now prove assertion \ref{non-minimal-types}. If $\bar X$ dominates $\pp^2$, then every singularity of $\bar X$ is sandwiched, as required. So assume that $\bar X^0$ does not dominate $\pp^2$. Let $\bar X^1$ be the normalization of the closure of the image of $\cc^2$ in $\bar X \times \pp^2$ defined via identification of $X$ with a copy of $\cc^2$ in $\pp^2$. Then all singularities of $\bar X^1$ are sandwiched. Assertion \ref{non-minimal-types} now follows from the observation that the natural projection $\bar X^1 \to \bar X$ is an isomorphism over the complement of the strict transform on $\bar X^1$ of the line at infinity on $\pp^2$. The last statement of assertion \ref{minimal-configuration} follows from similar reasoning and an application of \cite[Proposition 3.1]{sub2-2}.
\end{proof}

\section{Classification results for primitive compactifications}
\subsection{Primitive algebraic compactifications}
Using the correspondence with plane curves with one place at infinity (Theorem \ref{algebraic-place}), it is possible to explicitly describe the defining equations of all primitive algebraic compactifications of $\cc^2$. In particular, it turns out that every primitive algebraic compactification is a `weighted complete intersection' (embedded in a weighted projective variety). We now describe this result.

\begin{defn}[{\cite[Definition 3.2]{sub2-2}}] \label{key-seqn}
A sequence $\vec\omega := (\omega_0, \ldots, \omega_{n+1})$, $n \in \zz_{\geq 0}$, of positive integers is called a {\em key sequence} if it has the following properties: let $d_k := \gcd(\omega_0, \ldots, \omega_k)$, $0 \leq k \leq n+1$ and $p_k := d_{k-1}/d_k$, $1 \leq k \leq n+1$. Then
\begin{enumerate}
\item \label{gcd-1-property} $d_{n+1} = 1$, and
\item \label{smaller-property}  $\omega_{k+1} < p_k\omega_k$, $1 \leq k \leq n$.
\end{enumerate} 
A key sequence $(\omega_0, \ldots, \omega_{n+1})$ is called {\em algebraic} if in addition 
\begin{enumerate}
\setcounter{enumi}{2}
\item \label{semigroup-property} $p_k\omega_k \in \zz_{\geq 0} \langle \omega_0, \ldots, \omega_{k-1} \rangle$, $1 \leq k \leq n$.
\end{enumerate}
Finally, a key sequence $(\omega_0, \ldots, \omega_{n+1})$ is called {\em essential} if $p_k \geq 2$ for $1 \leq k \leq n$. Given an arbitrary key sequence $(\omega_0, \ldots, \omega_{n+1})$, it has an associated {\em essential subsequence} $(\omega_0, \omega_{i_1}, \ldots, \omega_{i_l}, \omega_{n+1})$ where $\{i_j\}$ is the collection of all $k$, $1 \leq k \leq n$, such that $p_k \geq 2$.
\end{defn}

\begin{rem} \label{approximate-remark}
Let $\bar X$ be a primitive algebraic compactification of $\cc^2$. Theorem \ref{projective-embedding} below states that $\bar X$ has an associated {\em algebraic key sequence} $\vec \omega$. On the other hand Theorem \ref{algebraic-place} attaches to $\bar X$ a curve $C$ with one place at infinity. It turns out that the essential subsequence $\vec \omega_e$ of $\vec \omega$ is `almost the same as' the {\em $\delta$-sequence} of $C$ (defined e.g.\ in \cite[Section 3]{suzuki}) - see \cite[Remark 2.10]{contractibility} for the precise relation. Moreover, recall (from Remark \ref{effective-remark}) that the {\em last key form} $g$ of the divisorial valuation associated to the curve at infinity on $\bar X$ is a polynomial and defines a curve $C$ as in the preceding sentence. Then it can be shown that the polynomials $G_1, \ldots, G_n$ (which induces an embedding of $\bar X$ into a weighted projective space) defined in Theorem \ref{projective-embedding} below contains a subsequence $G_{i_1}, \ldots, G_{i_l}$ such that $G_{i_j}|_{\cc^2}$ are precisely the {\em approximate roots} (introduced by Abhyankar and Moh \cite{abhya-moh-tschirnhausen}) of $g$. 
\end{rem}

\begin{rem}\label{unique-remark}
Let $\vec\omega := (\omega_0, \ldots, \omega_{n+1})$ be a key sequence. It is straightforward to see that property \ref{smaller-property} implies the following: for each $k$, $1 \leq k \leq n$, $p_k\omega_k$ can be {\em uniquely} expressed in the form $p_k\omega_k = \beta_{k,0}\omega_0 + \beta_{k,1} \omega_1 + \cdots + \beta_{k,k-1}\omega_{k-1}$, where $\beta_{k,j}$'s are integers such that $0 \leq \beta_{k,j} < p_j$ for all $j \geq 1$. $\beta_{k,0} \geq 0$. If $\vec \omega$ is in additional {\em algebraic}, then $\beta_{k,0}$'s of the preceding sentence are {\em non-negative}.
\end{rem}

\begin{thm}[{\cite[Proposition 3.5]{sub2-2}}] \label{projective-embedding}
Let $\vec\omega:=(\omega_0, \ldots, \omega_{n+1})$ be an algebraic key sequence. Let $w,y_0, \ldots, y_{n+1}$ be indeterminates. Pick $\theta_1, \ldots, \theta_n \in \cc^*$ and define polynomials $G_1, \ldots, G_n \in \cc[w,y_0, \ldots, y_{n+1}]$ as follows:
\begin{gather} \label{projective-equations}
G_k := w^{p_k\omega_k - \omega_{k+1}}y_{k+1} - \left( y_k^{p_k} - \theta_k \prod_{j=0}^{k-1}y_j^{\beta_{k,j}}\right)
\end{gather}
where $p_k$'s and $\beta_{k,j}$'s are as in Remark \ref{unique-remark}. Let $\bar X_{\vec \omega,\vec \theta}$ be the subvariety of the weighted projective space $\WP := \pp^{n+2}(1,\omega_0, \ldots, \omega_{n+1})$ (with weighted homogeneous coordinates $[w:y_0:y_1: \cdots :y_{n+1}]$) defined by $G_1, \ldots, G_n$. Then $\bar X_{\vec \omega,\vec \theta}$ is a primitive compactification of $\cc^2 \cong \bar X_{\vec \omega,\vec \theta} \setminus V(w)$. Conversely, every primitive algebraic compactification of $\cc^2$ is of the form $\bar X_{\vec \omega,\vec \theta}$ for some $\vec \omega, \vec \theta$.
\end{thm}

A more or less straightforward corollary is:

\begin{cor}[{\cite[Proposition 3.1, Corollary 3.6]{sub2-2}}]
Let $\bar X$ be a primitive algebraic compactification of $\cc^2$. Consider the equations of $\bar X$ from Proposition \ref{projective-embedding}. Let $C_\infty := \bar X \setminus X = \bar X \setminus V(w)$ and $P_\infty$ (resp.\ $P_0$) be the point on $C_\infty$ with coordinates $[0: \cdots :0:1]$ (resp.\ $[0:1: \bar \theta_1: \cdots : \bar \theta_n : 0]$), where $\bar \theta_k$ is an $p_k$-th root of $\theta_k$, $1 \leq k \leq n$). Then
\begin{enumerate}
\item $\bar X \setminus \{P_0, P_\infty\}$ is non-singular.
\item \label{non-wt-singularity} If $\bar X$ is not a weighted projective space, then $P_\infty$ is a singular point of $\bar X$. 
\item \label{cyclic} Let $\tilde \omega := \gcd(\omega_0,\ldots, \omega_n)$. Then $P_0$ is a cyclic quotient singularity of type $\frac{1}{\tilde \omega}(1,\omega_{n+1})$. 
\item \label{non-singinfinity} $C_\infty \setminus P_\infty \cong \cc$. In particular, $C_\infty$ is non-singular off $P_\infty$.
\item \label{curve-1} Let $S$ be the subsemigroup of $\zz^2$ generated by $\{(\omega_k,0): 0 \leq k \leq n\} \cup \{(0,\omega_{n+1})\}$. Then $C_\infty \cong \proj \cc[S]$, where $\cc[S]$ is the semigroup algebra generated by $S$, and the grading in $\cc[S]$ is induced by the sum of coordinates of elements in $S$.  
\item \label{curve-singularity} Let $\tilde S:= \zz_{\geq 0} \langle p_{n+1}\omega_{n+1} \rangle \cap \zz_{\geq 0} \langle \omega_0, \ldots, \omega_n \rangle$. Then $\cc[C_\infty \setminus P_0] \cong \cc[\tilde S]$,  In particular, $C_\infty$ has at worst a (non-normal) {\em toric singularity} at $P_\infty$.
\end{enumerate}
\end{cor}

Let $\bar X_{\vec \omega, \vec \theta}$ be an algebraic primitive compactification of $\cc^2$. We can compute the canonical divisor of $\bar X_{\vec \omega, \vec \theta}$ in terms of $\vec \omega$: 

\begin{thm}[{\cite[Theorem 4.1]{sub2-2}}] \label{canonical-thm}
Let $p_1, \ldots, p_{n+1}$ be as in the definition of algebraic key sequences. Then the canonical divisor of $\bar X_{\vec \omega, \vec \theta}$ is 
\begin{align}
        K_{\bar X_{\vec \omega, \vec \theta}} &= -\left(\omega_0 + \omega_{n+1} + 1 - \sum_{k=1}^n (p_k -1)\omega_k \right)[C_\infty], \label{canonical-formula}
\end{align}
where $[C_\infty]$ is the Weil divisor corresponding to $C_\infty$. Moreover, the {\em index} of $\xomegatheta$ (i.e.\ the smallest positive integer $m$ such that $mK_{\xomegatheta}$ is Cartier) is 
\begin{align}
\ind(K_{\xomegatheta}) = \min\left\{m\in\zz_{\geq 0}: m\left(\omega_0 + \omega_{n+1} + 1 - \sum_{k=1}^n (p_k -1)\omega_k \right) \in \zz p_{n+1} \cap \zz \omega_{n+1}\right\}.
\end{align}
\end{thm}

\subsection{Special types of primitive algebraic compactifications}
Straightforward applications of Theorem \ref{canonical-thm} yield the following characterizations of primitive algebraic compactifications of $\cc^2$ which have only rational or elliptic singularities, and those which are Gorenstein. For these results, let $\xomegatheta$ be, as in Theorem \ref{projective-embedding}, the primitive algebraic compactification corresponding to an algebraic key sequence $\vec \omega := (\omega_0, \ldots, \omega_{n+1})$ and $\vec\theta \in (\cc^*)^n$. 

\begin{cor}[Simple singularities, {\cite[Corollary 4.4]{sub2-2}}] \label{singularity-cor}
\mbox{}
\begin{enumerate}
\item \label{rational-assertion} $\xomegatheta$ has only rational singularities iff $\omega_0 + \omega_{n+1} + 1 - \sum_{k=1}^n (p_k -1)\omega_k > 0$.
\item $\xomegatheta$ has only elliptic singularities iff $0 \geq \omega_0 + \omega_{n+1} + 1 - \sum_{k=1}^n (p_k -1)\omega_k > - \omega_{\min}$, where $\omega_{\min} := \min\{\omega_0, \ldots, \omega_{l+1}\}$.
\end{enumerate}
\end{cor}

\begin{cor}[Gorenstein, {\cite[Proposition 4.5]{sub2-2}}] \label{gorenstein-cor}
$\xomegatheta$ is Gorenstein iff the following properties hold:
\begin{enumerate}
\item \label{gor1} $p_{n+1}$ divides $\omega_{n+1} + 1$, and
\item \label{gor2} $\omega_{n+1}$ divides $\omega_0 + \omega_{n+1} + 1 - \sum_{k=1}^n (p_k -1)\omega_k$. 
\end{enumerate}
\end{cor}

In the case that the anti-canonical divisor of $\xomegatheta$ is ample, a deeper examination of conditions \ref{gor1} and \ref{gor2} of Corollary \ref{gorenstein-cor} yields the following result which is originally due to \cite{brenton-singular} and \cite{brenton-graph-1}. We will use the following construction:

\begin{defn} \label{y-k}
For $1 \leq k \leq 8$, we now describe a procedure to construct a compactification $Y_k$ of $\cc^2$ via $n$ successive blow ups from $\pp^2$. We will denote by $E_k$, $1 \leq k \leq 8$, the $k$-th exceptional divisor on $Y_k$. Let $E_0$ be the line at infinity in $\pp^2$ and pick a point $O \in E_0$. Let $Y_1$ be the blow up of $\pp^2$ at $O$, and for $2 \leq k \leq 3$, let $Y_k$ be the blow up of $Y_{k-1}$ at the intersection of the strict transform of $E_0$ and $E_{k-1}$. Finally, for $3 \leq k \leq 7$, pick a point $O_k$ on $E_k$ which is not on the strict transform of any $E_j$, $0 \leq j \leq k - 1$, and define $Y_{k+1}$ to be the blow up of $Y_{k}$ at $O_k$.
\end{defn}

\begin{figure}[htp]
\begin{center}
\begin{tikzpicture}

\newcommand{\drawvertex}[6]{ 	
 	\fill[black] (#1, #2) circle (2pt);
 	\draw (#1, #2)  node [#5] {$#6$};
 	\draw (#1, #2)  node [#3] {$E_{#4}$};
}

\newcommand{\drawvertexplus}[6]{
	\draw[#1] (\x,\y) -- (\x+ #2,\y);
	\pgfmathsetmacro\x{\x + #2}
	\drawvertex{\x}{\y}{#3}{#4}{#5}{#6};	
}

\newcommand{\drawvertexplusy}[6]{
	\draw[#1] (\x,\y) -- (\x,\y- #2);
	\pgfmathsetmacro\y{\y - #2}
	\drawvertex{\x}{\y}{#3}{#4}{#5}{#6};	
}

 	\pgfmathsetmacro\edge{1}
 	\pgfmathsetmacro\vedge{.9}
 	\pgfmathsetmacro\dashededge{2}
 	\pgfmathsetmacro\dist{2}
 	\pgfmathsetmacro\captiony{-2*\vedge-0.75}
	\pgfmathsetmacro\captionx{0.5*\edge}

 	\pgfmathsetmacro\x{0}
 	\pgfmathsetmacro\y{0}
 	 	 	
 	\drawvertex{\x}{\y}{below}{0}{above}{0};
 	\drawvertexplus{thick}{\edge}{below}{1}{above}{-1};
 	\draw (\captionx, \captiony) node {(a) $k=1$};

 	\pgfmathsetmacro\x{\x+\dist}
 	\drawvertex{\x}{\y}{below}{0}{above}{-1};
 	\drawvertexplus{thick}{\vedge}{below right}{2}{above}{-1};
 	\drawvertexplusy{thick}{\vedge}{right}{1}{left}{-2};
 	\pgfmathsetmacro\y{0}
 	\pgfmathsetmacro\captionx{\captionx + \dist+\edge}
 	\draw (\captionx, \captiony) node {(b) $k=2$};
 	
 	\pgfmathsetmacro\x{\x+\dist}
 	\drawvertex{\x}{\y}{below}{0}{above}{-2};
 	\drawvertexplus{thick}{\edge}{below right}{3}{above}{-1};
 	\drawvertexplusy{thick}{\vedge}{right}{2}{left}{-2};
 	\drawvertexplusy{thick}{\vedge}{right}{1}{left}{-2};
 	\pgfmathsetmacro\y{0}
 	\pgfmathsetmacro\captionx{\captionx + \dist+\edge}
 	\draw (\captionx, \captiony) node {(c) $k=3$};
 	 	
 	\pgfmathsetmacro\x{\x+\dist}
 	\drawvertex{\x}{\y}{below}{0}{above}{-2};
 	\drawvertexplus{thick}{\edge}{below right}{3}{above}{-2};
 	\drawvertexplusy{thick}{\vedge}{right}{2}{left}{-2};
 	\drawvertexplusy{thick}{\vedge}{right}{1}{left}{-2};
 	\pgfmathsetmacro\y{0}
 	\drawvertexplus{thick, dashed}{\dashededge}{below}{k-1}{above}{-2};
 	\drawvertexplus{thick}{\edge}{below}{k}{above}{-1};
 	\pgfmathsetmacro\captionx{\captionx + \dist+2*\edge}
 	\draw (\captionx, \captiony) node {(d) $4\leq k \leq 8$};
 	
\end{tikzpicture}
\end{center}
\caption{Dual graph of curves at infinity on $Y_{k}$} \label{fig2}
\end{figure}
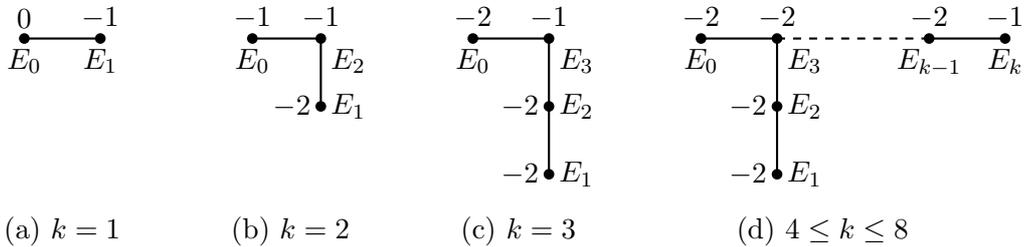

\begin{cor}[Gorenstein plus vanishing geometric genus, {\cite[Proposition 2]{brenton-singular}, \cite[Theorem 6]{brenton-graph-1}}] \label{neg-gorenstein}
Let $\bar X$ be a primitive Gorenstein compactification of $\cc^2$. Then the following are equivalent :
\begin{enumerate}[(i)]
\item the geometric genus $p_g(\bar X)$ of $\bar X$ is zero,
\item each singular point of $\bar X$ is a rational double point, and
\item the canonical bundle $K_{\bar X}$ of $\bar X$ is anti-ample. 
\end{enumerate} 
If any of these holds then one of the following holds:
\begin{enumerate}
\item $\bar X \cong \pp^2$,
\item $\bar X$ is the singular quadric hypersurface $x^2+y^2+z^2=0$ in $\pp^3$, or
\item $\bar X$ is obtained from some $Y_k$ (from Definition \ref{y-k}), $3 \leq k \leq 8$, by contracting the strict transforms of all $E_j$ for $0 \leq j < k$. 
\end{enumerate}  
In particular, if $\bar X$ is singular, then the dual curve for the resolution of singularities of $\bar X$ is one of the Dynkin diagrams $A_1$, $A_1 + A_2$, $A_4$, $E_5$, $E_6$, $E_7$ or $E_8$ (with the weight of each vertex being $-2$).  
\end{cor}

Miyanishi and Zhang in \cite{miyanishi-zhang} proved a converse to Corollary \ref{neg-gorenstein}. Recall that a surface $S$ is called {\em log del Pezzo} if $S$ has only quotient singularities and the anticanonical divisor $-K_S$ is ample.

\begin{thm}[{\cite[Theorem 1]{miyanishi-zhang}}] \label{converse-graph-0}
Let $S$ be a Gorenstein log del Pezzo surface of rank one. Then $S$ is a compactification of $\cc^2$ iff the dual curve for the resolution of singularities of $\bar X$ is one of the Dynkin diagrams $A_1$, $A_1 + A_2$, $A_4$, $E_5$, $E_6$, $E_7$ or $E_8$.
\end{thm}

In the same article Miyanishi and Zhang give a topological characterization of primitive Gorenstein compactification of $\cc^2$ with vanishing geometric genus:

\begin{thm}[{\cite[Theorem 2]{miyanishi-zhang}}] \label{converse-simply-connected-0}
Let $S$ be a Gorenstein log del Pezzo surface. Suppose that either $S$ is singular or that there are no $(- 1)$-curves contained in the smooth locus of $S$. Then $S$ is a compactification of $\cc^2$ iff the smooth locus of $S$ is simply connected. 
\end{thm}

From Theorem \ref{canonical-thm} and the classification of dual graph of resolution of singularities of primitive compactifications discussed in Section \ref{dual-subsection}, it is possible to obtain classifications of primitive compactifications with ample anti-canonical divisors and log terminal and log canonical singularities obtained originally by Kojima \cite{kojima} and Kojima and Takahashi \cite{koji-hashi}. Both of these classifications consist of explicit lists of dual graphs of resolution of singularities, and we omit their statements. However, they also prove converse results in the spirit of Theorems \ref{converse-graph-0} and \ref{converse-simply-connected-0}.         

\begin{thm}[{\cite[Theorem 0.1]{kojima}}]
Let $S$ be a log del Pezzo surface of rank one. Assume that the singularity type of $S$ is one of the possible choices (listed in \cite[Appendix C]{kojima}) for the singularity type of primitive compactifications of $\cc^2$ with at most quotient singularities. If $\ind(S) \leq 3$, then $S$ is a primitive compactification of $\cc^{2}$.
\end{thm}

\begin{thm}[{\cite[Theorem 1.2]{koji-hashi}}]
Let $S$ be a numerical del Pezzo surface (i.e.\ the intersection of the anti-canonical divisor of $S$ with itself and every irreducible curve on $S$ is positive) with at most rational singularities. Assume the singularity type of $S$ is one of the possible choices (listed in \cite{koji-hashi}) for the singularity type of primitive numerical del Pezzo compactifications of $\cc^2$ with rational singularities. Then $S$ is a primitive compactification of $\cc^{2}$.
\end{thm}

From a slightly different perspective, Furushima \cite{furushima} and Ohta \cite{ohta} studied primitive compactifications of $\cc^2$ which are hypersurfaces in $\pp^3$. The following was conjectured and proved for $d \leq 4$ by Furushima, and then proved in the general case by Ohta: 

\begin{thm}[{\cite{furushima}, \cite{furushima-errata}, \cite{ohta}}]\label{furujecture}
Let $\bar X_d$ be a minimal compactification of $\cc^2$ which is a hypersurface of degree $d \geq 2$ in $\pp^3$ and $C_d := \bar X_d \setminus \cc^2$ be the curve at infinity. Assume $\bar X_d$ has a singular point $P$ of multiplicity $d-1$. Then
\begin{enumerate}
\item $P$ is the unique singular point of $\bar X_d$ and the geometric genus of $P$ is $p_g(P) = (d - 1)( d - 2)(d - 3 )/6$. 
\item $C_d$ is a line on $\pp^3$.
\item $(\bar X_d, C_d) \cong (V_d, L_d)$ (up to a linear change of coordinates), where 
\begin{align*}
V_d &:= \{[z_0:z_1:z_2:z_3] \in \pp^3: z_0^d = z_1^{d-1}z_2 + z_2^{d-1}z_3\}\\
L_d &:= \{z_0= z_2 = 0\}.
\end{align*}
\end{enumerate}
\end{thm}

\subsection{Dual graphs for the resolution of singularities} \label{dual-subsection}
Let $\vec \omega := (\omega_0, \ldots, \omega_{n+1})$ be a key sequence. Then to every $\vec \theta := (\theta_1, \ldots, \theta_n) \in (\cc^*)^n$, we can associate a primitive compactification $\xomegatheta$ of $\cc^2$. Moreover, $\xomegatheta$ is algebraic iff $\vec \omega$ is an {\em algebraic} key sequence, and the correspondence $(\vec \omega, \vec \theta) \mapsto \xomegatheta$ is given by Theorem \ref{projective-embedding}. The correspondence in the general case is treated in \cite{sub2-1}; in our notation it can be described as follows: define $G_1, \ldots, G_n \in A := \cc[w,y_0,y_0^{-1}, y_1, \ldots, y_{n+1}]$ as in Theorem \ref{projective-embedding} (if $\vec \omega$ is not algebraic, then at least one of the $G_k$'s will not be a polynomial). Let $I$ be the ideal in $A$ generated by $w - 1, G_1, \ldots, G_n$. Then $A/I \cong \cc[x,x^{-1},y]$ via the map $y_0 \mapsto x, y_1 \mapsto y$. Let $f_k \in \cc[x,x^{-1},y]$ be the image of $G_k$, $1 \leq k \leq n$. Consider the family of curves $C_\xi \subseteq \cc^2 \setminus V(x)$, $\xi \in \cc$, defined by $f_n^{\omega_0} = \xi x^{\omega_{n+1}}$. Then $\xomegatheta$ is the unique primitive compactification of $\cc^2 = \spec \cc[x,y]$ which {\em separates (some branches of) the curves $C_\xi$ at infinity}, i.e.\ for generic $\xi$, the closure of the curve $C_\xi$ in $\xomegatheta$ intersects generic points of the curve at infinity. It follows from the results of \cite[Corollary 4.11]{sub2-1} that every primitive compactification of $\cc^2$ is of the form $\xomegatheta$ for some appropriate $\vec \omega$ and $\vec\theta$. 

\begin{rem}[A valuation theoretic characterization of $\xomegatheta$] \label{general-key-remark}
Let $f_1, \ldots, f_n$ be as in the preceding paragraph. Then $\xomegatheta$ is the unique primitive compactification of $\cc^2 = \spec \cc[x,y]$ such that the {\em key forms} (see Remark \ref{effective-remark}) of the valuation on $\cc[x,y]$ corresponding to the curve at infinity on $\xomegatheta$ are $x,y, f_1, \ldots, f_n$.
\end{rem}

The dual graph of the minimal resolution of singularities of $\xomegatheta$ depends only on the {\em essential subsequence} (Definition \ref{key-seqn}) $\vec \omega_e$ of $\vec \omega$. The precise description of the dual graph in terms of $\vec \omega_e$ is a bit technical and it essentially corresponds to the resolution of singularities of a point at infinity on (the closure of) the curve $C_\xi$ from the preceding paragraph for generic $\xi$ - we refer to \cite[Appendix]{sub2-1} for details. Rather we now state the characterization from \cite{contractibility} of those dual graphs which appear only for algebraic, only for non-algebraic, and for both algebraic and non-algebraic compactifications.

\begin{thm}[{\cite[Theorem 2.8]{contractibility}}] \label{semigroup-prop}
Let $\vec \omega := (\omega_0, \ldots, \omega_{n+1})$ be an {\em essential} key sequence and let $\Gamma_{\vec \omega}$ be the dual graph for the minimal resolution of singularities for some (and therefore, every) primitive compactification $\bar X_{\vec \omega',\vec \theta}$ of $\cc^2$ where $\vec \omega'$ is a key sequence with essential subsequence $\vec \omega$. Then
\begin{enumerate}
\item \label{algebraic-existence} There exists a primitive {\em algebraic} compactification $\bar X$ of $\cc^2$ such that the dual graph for the minimal resolution of singularities of $\bar X$ is $\Gamma_{\vec \omega}$ iff $\vec \omega$ is an {\em algebraic} key sequence.
\item \label{non-algebraic-existence} There exists a primitive {\em non-algebraic} compactification $\bar X$ of $\cc^2$ such that the dual graph for the minimal resolution of singularities of $\bar X$ is $\Gamma_{\vec \omega}$ iff 
\begin{enumerate}
\item either $\vec \omega$ is {\em not} {\em algebraic}, or 
\item $\bigcup_{1 \leq k \leq n}\{ \alpha \in \zz\langle \omega_0, \ldots, \omega_k \rangle \setminus \zz_{\geq 0}\langle \omega_0, \ldots, \omega_k \rangle: \omega_{k+1} < \alpha < p_k\omega_k\} \neq \emptyset$.
\end{enumerate}
\end{enumerate}
\end{thm}

\begin{example}[{\cite[Corollary 2.13, Example 2.15]{contractibility}}]\label{non-2-remexample}
The dual graph of the minimal resolution of singularities of $\bar X_{i,r}$ from Example \ref{non-example} corresponds to the essential key sequence $(2,5)$ for $r=0$ and $(2,5,10-r)$ for $1 \leq r \leq 9$. 
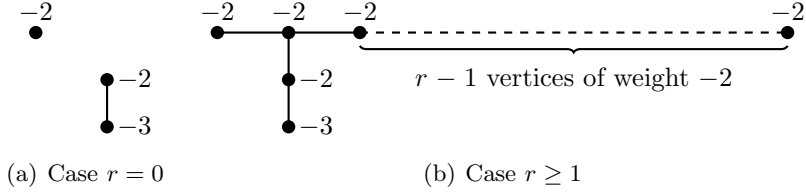
\begin{figure}[htp]
\begin{center}
\subfigure[Case $r = 0$]{
\begin{tikzpicture}[scale=1.25, font = \small] 	
 	\pgfmathsetmacro\edge{.75}
 	\pgfmathsetmacro\vedge{.5}
	
 	\draw[thick] (0,-\vedge) -- (0,-2*\vedge);

 	\fill[black] (-\edge, 0) circle (2pt);
 	\fill[black] (0, -\vedge) circle (2pt);
 	\fill[black] (0, - 2*\vedge) circle (2pt);
 	
 	\draw (-\edge,0 )  node (e1up) [above] {$-2$};	
 	\draw (0,-\vedge )  node (down1) [right] {$-2$};
 	\draw (0, -2*\vedge)  node (down2) [right] {$-3$};
 			 	
\end{tikzpicture}
}
\subfigure[Case $r \geq 1$]{
\begin{tikzpicture}[scale=1.25, font = \small] 	
 	\pgfmathsetmacro\dashedge{4.5}	
 	\pgfmathsetmacro\edge{.75}
 	\pgfmathsetmacro\vedge{.5}

 	\draw[thick] (-\edge,0) -- (\edge,0);
 	\draw[thick] (0,0) -- (0,-2*\vedge);
 	\draw[thick, dashed] (\edge,0) -- (\edge + \dashedge,0);
 	
 	\fill[black] (-\edge, 0) circle (2pt);
 	\fill[black] (0, 0) circle (2pt);
 	\fill[black] (0, -\vedge) circle (2pt);
 	\fill[black] (0, - 2*\vedge) circle (2pt);
 	\fill[black] (\edge, 0) circle (2pt);
 	\fill[black] (\edge + \dashedge, 0) circle (2pt);
 	
 	\draw (-\edge,0 )  node (e1up) [above] {$-2$};
 	\draw (0,0 )  node (middleup) [above] {$-2$};	
 	\draw (0,-\vedge )  node (down1) [right] {$-2$};
 	\draw (0, -2*\vedge)  node (down2) [right] {$-3$};
 	\draw (\edge,0)  node (e+1-up) [above] {$-2$};
 	\draw (\edge + \dashedge,0)  node (e-last-1-up) [above] {$-2$};
 	
 	\draw [thick, decoration={brace, mirror, raise=5pt},decorate] (\edge,0) -- (\edge + \dashedge,0);
 	\draw (\edge + 0.5*\dashedge,-0.5) node [text width= 5cm, align = center] (extranodes) {$r-1$ vertices of weight $-2$};
 			 	
\end{tikzpicture}
}
\caption{Dual graph of minimal resolution of singularities of $\bar X_{i,r}$ from Example \ref{non-example}}\label{fig:non-resolution}
\end{center}
\end{figure}
A glance at Table \ref{key-table} shows that $(2,5)$ and $(2,5,10-r)$, $1 \leq r \leq 9$, are algebraic key sequences, so that Theorem \ref{semigroup-prop} implies that each of these sequences corresponds to some algebraic primitive compactifications. Now note that for $\vec \omega := (2,5,10-r)$,  
$$\zz\langle \omega_0, \omega_1 \rangle \setminus \zz_{\geq 0}\langle \omega_0, \omega_1 \rangle = \zz\langle 2,5 \rangle\setminus \zz_{\geq 0}\langle 2,5 \rangle = \zz \setminus \zz_{\geq 0}\langle 2,5 \rangle = \{1,3\}.$$
Since for $r = 8,9$, we have $\omega_2 = 10 - r< 3$, Theorem \ref{semigroup-prop} implies that in this case $\vec \omega$ also corresponds to some non-algebraic primitive compactifications. In summary, $(2,5)$ and $(2,5,10-r)$, $1 \leq r \leq 7$, correspond to {\em only algebraic} primitive compactifications, and $(2,5,10-r)$, $8 \leq r \leq 9$, corresponds to {\em both} algebraic and non-algebraic compactifications, as it was shown in Example \ref{non-example}. \\

\begin{table}[htp]
\begin{tabular}{|>{$}l<{$}|>{$}l<{$}|>{$}l<{$}|>{$}l<{$}|}
\hline
(\omega_0, \ldots, \omega_{n+1}) & (d_0, \ldots, d_{n+1}) & (p_1, \ldots, p_{n+1}) & (p_1\omega_1, \ldots, p_{n}\omega_n) \\
\hline
(2,5) & (2,1) & (2) & \emptyset\\
(2,5,10-r) & (2,1,1) & (2,1) & (10)\\
(4,10,3,2) & (4,2,1,1) & (2,2,1) & (20,6)\\
\hline
\end{tabular}
\caption{Some key sequences $\vec \omega$ and corresponding $\vec d$, $\vec p$} \label{key-table}
\end{table}

On the other hand, for $\vec \omega = (4,10,3,2)$, Table \ref{key-table} shows that $p_2\omega_2 = 6 \not\in \zz_{\geq 0}\langle 4, 10 \rangle = \zz_{\geq 0}\langle \omega_0, \omega_1 \rangle$, so that $\vec \omega$ is {\em not} an algebraic key sequence. Consequently Theorem \ref{semigroup-prop} implies that $\Gamma_{(4,10,3,2)}$ corresponds to {\em only non-algebraic} primitive compactifications (see Figure \ref{fig:non-2-resolution}).

\begin{figure}[htp]

\begin{center}
\begin{tikzpicture}
        \pgfmathsetmacro\dashedge{4.5}  
        \pgfmathsetmacro\edge{.75}
        \pgfmathsetmacro\vedge{.5}
        
        \pgfmathsetmacro\leftone{5*\edge}
        \pgfmathsetmacro\labelpos{-4*\vedge}

        \draw[thick] (-\edge,0) -- (\edge,0);
        \draw[thick] (0,0) -- (0,-2*\vedge);
        \draw[thick, dashed] (\edge,0) -- (\edge + \dashedge,0);
        \draw[thick] (\edge + \dashedge,0) -- (3*\edge + \dashedge,0);
        \draw[thick] (3*\edge + \dashedge,0) -- (3*\edge + \dashedge,-\vedge);
        
        \fill[black] (-\edge, 0) circle (2pt);
        \fill[black] (0, 0) circle (2pt);
        \fill[black] (0, -\vedge) circle (2pt);
        \fill[black] (0, - 2*\vedge) circle (2pt);
        \fill[black] (\edge, 0) circle (2pt);
        \fill[black] (\edge + \dashedge, 0) circle (2pt);
        \fill[black] (2*\edge + \dashedge, 0) circle (2pt);
        \fill[black] (3*\edge + \dashedge, 0) circle (2pt);
        \fill[black] (3*\edge + \dashedge, -\vedge) circle (2pt);
        
        \draw (-\edge,0 )  node (e1up) [above] {$-2$};
        \draw (0,0 )  node (middleup) [above] {$-2$};   
        \draw (0,-\vedge )  node (down1) [right] {$-2$};
        \draw (0, -2*\vedge)  node (down2) [right] {$-3$};
        \draw (\edge,0)  node (e+1-up) [above] {$-2$};
        \draw (\edge + \dashedge,0)  node (e-last-1-up) [above] {$-2$};
        \draw (2*\edge + \dashedge,0)  node [above] {$-3$};
        \draw (3*\edge + \dashedge,0)  node [above] {$-2$};
        \draw (3*\edge + \dashedge,-\vedge)  node [right] {$-2$};
        
        \draw [thick, decoration={brace, mirror, raise=5pt},decorate] (\edge,0) -- (\edge + \dashedge,0);
        \draw (\edge + 0.5*\dashedge,-0.5) node [text width= 5cm, align = center] (extranodes) {$7$ vertices of weight $-2$};
                                
\end{tikzpicture}
\caption{$\Gamma_{(4,10,3,2)}$}\label{fig:non-2-resolution}
\end{center}
\end{figure}
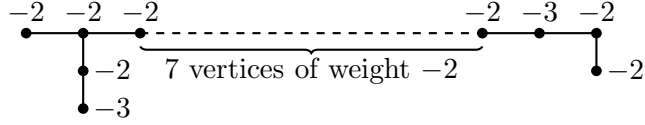
\end{example}

\section{Groups of automorphism and moduli spaces of primitive compactifications}
In \cite[Section 5]{sub2-2} the groups of automorphisms and moduli spaces of primitive compactifications have been precisely worked out. Here we omit the precise statements and content ourselves with the description of some special cases.

\begin{defn}
A key sequence $\vec \omega = (\omega_0, \ldots, \omega_{n+1})$ is in the {\em normal form} iff 
\begin{enumerate}
\item either $n = 0$, or 
\item $\omega_0$ does not divide $\omega_1$ and $\omega_1/\omega_0 > 1$.
\end{enumerate} 
\end{defn} 

\begin{thm}[{cf.\ \cite[Corollary 5.4]{sub2-2}}] \label{aut-thm}
Let $\bar X$ be a primitive compactification of $\cc^2$. Then $\bar X \cong \xomegatheta$ for some key sequence $\vec\omega := (\omega_0, \ldots, \omega_{n+1})$ in the normal form (and some appropriate $\vec\theta$). Moreover,
\begin{enumerate}
\item $n = 0$ iff $\bar X$ is isomorphic to some weighted projective space $\pp^2(1,p,q)$.
\item If $\bar X \not\cong \pp^2(1,1,q)$ for any $q \geq 1$, then there are coordinates $(x,y)$ on $\cc^2$ such that for every automorphism $F$ of $\bar X$, $F|_{\cc^2}$ is of the form $(x,y) \mapsto (ax + b, a'y + f(x))$ for some $a,a',b\in \cc$ and $f \in \cc[x]$. Moreover, if $n > 1$ then $a$ and $a'$ are some roots of unity and $b = f = 0$. 
\end{enumerate}
\end{thm}

\begin{thm}[{cf.\ \cite[Corollary 5.8]{sub2-2}}] \label{moduli-thm}
Let $\vec\omega := (\omega_0, \ldots, \omega_{n+1})$ be an essential key sequence in the normal form and $\scrxomega$ (resp.\ $\scrxomega^{alg}$) be the space of normal analytic (resp.\ algebraic) surfaces which are isomorphic to $\bar X_{\vec \omega', \vec \theta}$ for some key sequence $\vec \omega'$ with essential subsequence $\vec \omega$ and some $\vec\theta$. Then
\begin{enumerate}
\item $\scrxomega$ is of the form $\left((\cc^*)^k \times \cc^l\right)/G$ for some subgroup $G$ of $\cc^*$. 
\item $\scrxomega^{alg}$ is either empty (in the case that $\vec\omega$ is not algebraic), or a closed subset of $\scrxomega$ of the form $\left((\cc^*)^k \times \cc^{l'}\right)/G$ for some $l' \leq l$.
\end{enumerate}
\end{thm}

\begin{rem}
The correspondence of Theorem \ref{algebraic-place} between primitive algebraic compactifications with $\cc^2$ and planar curves with one place at infinity extends to their moduli spaces. The moduli spaces of curves with one place at infinity are of the form $(\cc^*)^k \times \cc^l$ for some $k,l\geq 0$ \cite[Corollary 1]{fujiki}. The extra complexity (i.e.\ action by the group $G$ from Theorem \ref{moduli-thm}) in the structure of the moduli spaces of primitive algebraic compactifications comes from the action of their groups of automorphisms. 
\end{rem}

Using (the precise version in \cite[Corollary 5.8]{sub2-2} of) Theorem \ref{aut-thm} it can be shown that $\pp^2$ is the only normal analytic surface of Picard rank $1$ which admits a $\mathbb{G}^2_a$ action with an open orbit \cite[Corollary 6.2]{sub2-2}. 

\bibliographystyle{alpha}
\bibliography{../../utilities/bibi}


\end{document}